\documentclass{amsart}

\usepackage{amsthm}
\usepackage{amssymb}
\usepackage{mathtools}
\usepackage{pdfpages}
\usepackage{changes}
\usepackage{listings}
\usepackage{tikz-cd}

\newcommand{\N}{\mathbb{N}}

\newcommand{\Sym}{\mathbb{S}}
\newcommand{\Alt}{\mathbb{A}}
\newcommand{\Aut}{\operatorname{Aut}}

\newcommand{\SL}{\mathrm{SL}}

\newcommand{\id}{\mathrm{id}}
\newcommand{\Soc}{\mathrm{Soc}}

\makeatletter
\numberwithin{equation}{section}
\numberwithin{figure}{section}
\numberwithin{table}{section}
\theoremstyle{plain}
\newtheorem{thm}{Theorem}[section]
\newtheorem*{thm*}{Theorem}
\theoremstyle{plain}
\newtheorem{lem}[thm]{Lemma}
\newtheorem{cor}[thm]{Corollary}
\theoremstyle{plain}
\newtheorem{pro}[thm]{Proposition}
\newtheorem{defn}[thm]{Definition}
\newtheorem*{conjecture*}{Conjecture}

\newtheorem{example}[thm]{Example}
\newtheorem{question}[thm]{Question}
\theoremstyle{remark}
\newtheorem{rem}[thm]{Remark}
\theoremstyle{plain}
\newtheorem{exa}[thm]{Example}

\makeatother

\begin{document}

\lstset{language=GAP,
  showstringspaces=false,
  xleftmargin=0cm,
  xrightmargin=0.6cm,
  basicstyle=\small\ttfamily,
  frame=single,
  framerule=0pt,
}

\title{Experimenting with braces}
\title{On the number and properties of small braces}
\title{On skew braces and their ideals}

\begin{abstract}
We define combinatorial representations of finite skew braces and use this idea
to produce a database of skew braces of small size. This database is then used
to explore different concepts of the theory of skew braces such as ideals, series of ideals, 
prime and semiprime ideals, Baer and Wedderburn radicals and solvability. 
The paper contains several questions.
\end{abstract}


\author{A. Konovalov}
\author{A. Smoktunowicz}
\author{L. Vendramin}

\address{Centre for Interdisciplinary Research in Computational Algebra,
University of St Andrews, North Haugh, St Andrews, Fife, KY16 9SX, UK}
\email{alexander.konovalov@st-andrews.ac.uk}

\address{School of Mathematics, The University of Edinburgh, James Clerk Maxwell Building, The Kings Buildings, Mayfield Road EH9 3JZ, Edinburgh, UK}
\email{A.Smoktunowicz@ed.ac.uk}

\address{IMAS--CONICET and Departamento de Matem\'atica, FCEN, Universidad de Buenos Aires, Pabell\'on~1, Ciudad Universitaria, C1428EGA, Buenos Aires, Argentina}
\email{lvendramin@dm.uba.ar}

\maketitle
\setcounter{tocdepth}{1}
\tableofcontents

\section*{Introduction}

In this work we explore some algebraic structures related to solutions to the
celebrated Yang--Baxter equation.  Following Drinfeld~\cite{MR1183474}, a
\emph{set-theoretic solution of the Yang--Baxter equation} is defined as a pair
$(X,r)$, where $X$ is a set and $r\colon X\times X\to X\times X$ is a bijection
such that
\[
r_1r_2r_1=r_2r_1r_2,
\quad r_1=r\times\id,
\quad r_2=\id\times r.
\]

We will be interested in \emph{non-degenerate} solutions, that is solutions
$(X,r)$ where $r$ can be written as $r(x,y)=(\sigma_x(y),\tau_y(x))$ for
permutations $\sigma_x$ and $\tau_x$ of $X$.

Rump found that there is a deep connection between radical rings and
set-theoretic solutions of the Yang--Baxter equation.  The key observation is
the following. Let $R$ be a radical ring, that is an associative ring $R$ such
that for each $x \in R$ there exists $y\in R$ such that $x + y + xy=0$.  Then
the operation $x\circ y=x+y+xy$ turns $R$ into a group and 
\[  
    r\colon R\times R\to R\times R,
    \quad
    r(x,y)=(xy+y,(xy+y)'\circ x\circ y),
\]
where $z'$ denotes the inverse of $z$ with respect to the circle operation
$\circ$, is a non-degenerate solution
of the Yang--Baxter equation such that 
$r^2=\id_{R\times R}$. 
A natural question arises: do we really need
radical rings to construct such solutions? 

In~\cite{MR2278047} Rump introduced braces, a generalization of radical rings
that produces involutive solutions. There is a rich theory of braces, see for
example~\cite{MR3320237,
	MR3465351,MR3763276,MR3527540,BCJO,MR3478858,CDS,MR3574204,MR3177933,MR3447734,Dietzel,
	GI15,Rump,MR2298848,MR3291816}.  Later braces were
	generalized to \emph{skew braces} to allow the construction of
	non-involutive solutions~\cite{MR3647970}.  A \emph{skew brace} is a triple
	$(A,\circ,+)$, where $(A,+)$ and $(A,\circ)$ are (not necessarily abelian)
	groups and the compatilibity condition
\[
	a\circ(b+c)=a\circ b-a+a\circ c
\]
holds for all $a,b,c\in A$.  

If $\mathcal{X}$ is a property of groups, a skew
brace is said to be of $\mathcal{X}$-type if its additive group belongs to
$\mathcal{X}$.  For example, skew braces of abelian type are those braces
introduced by Rump in \cite{MR2278047} to study involutive set-theoretic
solutions. Such braces will be also called either \emph{classical braces} or
\emph{braces}. 

Skew braces have connections to several
different topics, see for
example~\cite{Bachiller3,Brz,MR3649817,Childs,DeCommer,JVA,MR3763907}.
In particular, skew braces
provide the right algebraic framework to study set-theoretic solutions to the
Yang--Baxter equation. 
The connection between set-theoretic solutions and skew braces is explained in
the following theorems. The first one shows that skew braces produce
set-theoretic solutions:

\begin{thm*}{\cite[Theorem 3.1]{MR3647970}}
	Let $A$ be a skew brace. The map 
	\[
		r_A\colon A\times A\to A\times A,\quad
		r_A(a,b)=(-a+a\circ b,(-a+a\circ b)'\circ a\circ b),
	\]
	is a non-degenerate
	set-theoretic solution of the Yang--Baxter equation.
\end{thm*}

The second theorem shows that solutions associated to skew braces are, in some
sense, universal. Similar results are~\cite[Theorem 2.9]{MR1722951} for
involutive solutions, and~\cite[Theorem 9]{MR1769723} and~\cite[Theorem
2.7]{MR1809284} for non-involutive solutions.  Recall that the \emph{structure
group} of a solution $(X,r)$ is the group $G(X,r)$ generated by $\{x:x\in X\}$
with relations $xy=uv$ whenever $r(x,y)=(u,v)$.

\begin{thm*}{\cite[Theorem 4.5]{MR3763907}}
  Let $(X,r)$ be a non-degenerate solution of the Yang--Baxter equation. 
  Then there exists a unique skew left brace structure over the group $G(X,r)$ such
  that 
  \[
	  (\iota\times\iota)r=r_{G(x,r)}(\iota\times\iota),
  \]
where $\iota\colon X\to G(X,r)$ is the canonical map. 
Moreover, the pair $(G(X,r),\iota)$ has the following universal property: if $B$
is a skew left brace and $f\colon X\to B$ is a map such that 
$(f\times f)r=r_B(f\times f)$, 
then there exists a unique skew brace homomorphism
$\phi\colon G(X,r)\to B$ such that 
$f=\phi\iota$ and 
$(\phi\times\phi)r_{G(X,r)}=r_B(\phi\times\phi)$.  
\end{thm*}

This theorem allows us to define $G(X,r)$ as the structure skew brace of the
solution $(X,r)$.  Clearly, skew braces are useful for understanding
non-degenerate set-theoretic solutions of the Yang--Baxter equation. Moreover,
to study finite solutions one only needs finite skew braces, see~\cite[Theorem
3.11]{Bachiller3}. Hence, since skew braces generalize radical rings, tools and
ideas from ring theory can be used to study the Yang--Baxter equation.  

Braces and skew braces have a strong connection with
regular subgroups, see for example~\cite[Proposition
2.3]{MR3465351},~\cite[Theorem 1]{MR2486886} and~\cite[Theorem 4.2]{MR3647970}.
Based on this fact, an algorithm for constructing all skew braces of a given
size was developed in~\cite{MR3647970}.  Using it, one produces a huge database of all (skew)
braces of a given order. 

The first and the third author produced the~\textsf{GAP}
package~\textsf{YangBaxter} that implements several methods for studying skew
braces and other structures related to the set-theoretic Yang--Baxter equation. 
The package 
contains a database of classical and skew braces of small orders and 
it is freely available at
\verb+http://gap-packages.github.io/YangBaxter/+.

\medskip
The paper is organized as follows. In Section~\ref{representations} we
introduce combinatorial representations of skew braces; this concept is needed
to store small skew braces in a database.  In Sections~\ref{ideals}
and~\ref{series} we study ideals and some particular series of ideals of skew
braces; these sections contain several examples that answer some natural
questions.  Section~\ref{prime} is devoted to study prime and semiprime ideals
and related concepts such as the Baer radical and the Wedderburn radical of a
skew brace. This section contains some of our main results. In
Theorem~\ref{thm:B(A)=0<=>semiprime} we prove that a skew brace is semiprime if
and only if its Baer radical is zero.  Theorem~\ref{thm:Baer} proves that the
Baer radical of a skew brace is the intersection of all its prime ideals. In
Theorem~\ref{thm:subdirect} we prove that every semiprime skew brace is a
subdirect product of prime skew braces. A relation between the Wedderburn and
the Baer radical is stated in Theorem~\ref{thm:B(A)=0<=>W(A)=0}. 
Solvable ideals of skew braces are studied
in Section~\ref{solvable}. One of our main results is Theorem~\ref{main}, where
it is proved that a finite skew brace is solvable if and only if it is Baer
radical.

\section{Combinatorial representations of finite skew braces}
\label{representations}

When storing skew braces in a database, an obvious question is how to represent
them efficiently. Obviously, each skew brace can be given
by the tables for addition and multiplication, but that would cause a substantial
overhead. On the other hand, one can substantially reduce the storage size by
keeping only generators for the additive and multiplicative groups of a skew brace,
and recording a way to reconstruct its full structure. This process should be 
deterministic and should not depend on some randomized algorithms. If we
store additive and multiplicative groups as permutation groups, we can rely on
the lexicographic ordering of permutations and store skew braces as explained below.

\begin{defn}
\label{rem:order}
For permutations $f$ and $g$, 
$f<g$ 
if and only if the image of $f$ on the range from $1$ to the degree of $f$ is lexicographically smaller than the corresponding image for $g$. 
\end{defn}

Recall from ~\cite[Proposition 1.11]{MR3647970} that a skew brace of size $n$ 
with additive group $A$ is equivalent to a pair $(G,\pi)$ where $G$ is a group acting by automorphisms on $A$ 
and $\pi \colon G \to A$ is a bijective $1$-cocycle.
Without loss of generality we can write 
$G=\{g_1,g_2,\dots,g_n\}$ and $A=\{a_1,a_2,\dots,a_n\}$ as permutation groups 
and assume that 
$\pi(g_j)=a_j$ for all $j\in\{1,\dots,n\}$. Then the skew brace
is the additive group $A=\{a_1,\dots,a_n\}$ with
the multiplication 
\[
a_i\circ a_j=a_k,
\]
where $g_ig_j=g_k$. This means that to store our skew brace we only need
these two tuples of permutations $(a_1,a_2,\dots,a_n)$ and $(g_1,g_2,\dots,g_n)$.
Observe the use of tuples is very important because it implies that elements
of $G$ and $A$ are listed in a particular order, determined by the bijection $\pi$.

This way, we will need $2n$ permutations to store a brace of size $n$. We can try to be 
more efficient by storing generating sets of groups $G$ and $A$, together with the data needed
to recover the tuples $(a_1,a_2,\dots,a_n)$ and $(g_1,g_2,\dots,g_n)$.
To recover these tuples, first we use an algorithm that constructs the lists of all elements of the
groups $G$ and $A$ from the chosen generating sets, and then sort each of the resulting lists
in lexicographic order (see Definition~\ref{rem:order}).
So we obtain
\[
a_{\sigma(1)}<a_{\sigma(2)}<\cdots<a_{\sigma(n)},\quad
g_{\tau(1)}<g_{\tau(2)}<\cdots<g_{\tau(n)},
\]
where $\sigma$ and $\tau$ are some permutations of $\{1,\dots,n\}$. These translate
into two tuples $( a_{\sigma(1)}, a_{\sigma(2)}, \cdots , a_{\sigma(n)} )$ and
$(g_{\tau(1)},g_{\tau(2)},\cdots,g_{\tau(n)})$.
Acting with the inverses of $\sigma$ and $\tau$ we recover the 
tuples $(a_1,a_2,\dots,a_n)$ and $(g_1,g_2,\dots,g_n)$ respectively.

Note that the generating sets of $G$ and $A$
do not have to be of a minimal size, although for practical purposes it is
useful to choose them as small as possible.

\subsection*{A database of small skew braces}

Motivated
by~\cite{MR1935567} and 
using the algorithm described in~\cite{MR3647970}, one constructs a database of
small (skew) braces.  Thanks to the representation described in the previous
section, we were able to reduce the size of the database from more than 300 MB
in the initial representation (which kept full lists of elements of permutation
representation of the additive and multiplicative group of a skew brace) to
less than 30 MB. 

At the present moment,
the database contains all (up to isomorphism) skew braces of sizes up to 85 except
some orders including large prime powers, e.g. 32, 64, etc.
and all (up to isomorphism) classical braces of sizes up to 127 except 32, 64, 81 and 96.
In total, it included 96830 skew braces and 8828 classical braces. 

Each classical brace (respectively skew brace) is
named by their library index as $B_{n,k}$ (resp. $S_{n,k}$), where $n$ is its
size and $k$ is its index in the database of braces of size $n$. For example,
the list of skew braces of size eight is $S_{8,1}, S_{8,2}, \dots, S_{8,47}$, and the list of 
classical ones is $B_{8,1}, B_{8,2}, \dots, B_{8,27}$.


The number $s(n)$ of isomorphism classes of skew braces and $b(n)$ of classical braces 
for $n \le 16$ is given in Table~\ref{brace-db}.

\begin{table}
\caption{Number of skew and classical braces for  $n \le 16$.}
\label{brace-db}
\begin{tabular}{|r|cccccccc|}
\hline 
$n$ & 1 & 2 & 3 & 4 & 5 & 6 & 7 & 8\tabularnewline
$s(n)$ & 1 & 1 & 1 & 4 & 1 & 6 & 1 & 47\tabularnewline
$b(n)$ & 1 & 1 & 1 & 4 & 1 & 2 & 1 & 27\tabularnewline
\hline 
$n$ & 9 & 10 & 11 & 12 & 13 & 14 & 15 & 16\tabularnewline
$s(n)$ & 4 & 6 & 1 & 38 & 1 & 6 & 1 & 1605\tabularnewline
$b(n)$ & 4 & 2 & 1 & 10 & 1 & 2 & 1 & 357\tabularnewline
\hline 
\end{tabular}
\end{table}


\subsection*{An application to two-sided skew braces}
\label{twosided}

In~\cite[Question 2.1(2)]{CDS} one finds the following interesting question: 
Is it true that any brace such that the operation $a*b=-a+a\circ b-b$ is associative
is a two-sided brace? 
We check that the answer is affirmative for all the classical braces of
our database.  We know from~\cite[Proposition 2.2]{CDS} that we only need to
check classical braces of even size. We have tested
all such classical braces in our database and we found no answer to this question.

%
%
%
%

What happens if we ask the same question for skew braces?  Now it turns out
that indeed we have an answer! The smallest skew braces which are not two-sided
and have an associative $*$ operation are
\[
S_{16,j},\quad
j\in\{230,235,424,429,547,554,556,561\}.
\]
It is interesting to observe that the additive groups of these skew braces are
nilpotent.  Since skew braces with nilpotent additive groups are almost like
classical braces, these examples of size 16 suggest that one should expect an
answer to Question~\cite[Question 2.1(2)]{CDS} in the positive.

%

\section{Ideals of skew braces}
\label{ideals}

Since skew braces are generalizations of radical rings, one can try to
exploit ideas from ring theory. Let us first recall a very useful lemma:

\begin{lem}
\label{lambda}
Let $A$ be a skew brace. Then $\lambda\colon (A,\circ)\to\Aut(A,+)$ given by
$a\mapsto\lambda_a$, where $\lambda_a(b)=-a+a\circ b$, is a well-defined group homomorphism.
\end{lem}

\begin{proof}
    See~\cite[Corollary 1.10]{MR3763276}.
\end{proof}

An \emph{ideal} of a skew brace $A$ is a normal
subgroup $I$ of the multiplicative group of $A$ such that
$\lambda_a(I)\subseteq I$ and $a+I=I+a$ for all $a\in A$. The following easy 
lemma is useful for computational purposes:

\begin{lem}
\label{lem:ideals}
	Let $A$ be a skew brace and $I$ be a subset of $A$. Then $I$ is an ideal if
	and only if $I$ is a normal subgroup of the additive group of $A$, $a\circ
	I=I\circ a$ and $\lambda_a(I)\subseteq I$ for all $a\in A$.
\end{lem}

\begin{proof}
    Assume first that $I$ is an ideal of $A$. Then the claim follows from \cite[Lemma 2.3(1)]{MR3647970}.
    To prove the converse we need to show that 
    $I$ is a subgroup of the multiplicative group of $A$. For $x,y\in I$, 
    \[
        x\circ y'=x\lambda_x(-\lambda_{y'}(y))\in I
    \]
    and hence the claim follows. 
\end{proof}

The \emph{socle} of a skew brace $A$ is defined as 
$\Soc(A)=\ker\lambda\cap Z(A,+)$ and it is an ideal of $A$. 
A skew brace  $A$ is said to be \emph{trivial} if $a+b=ab$ for all $a,b\in A$.


\begin{exa}
	Let $A=S_{6,1}$, the trivial skew brace over $\Sym_3$. 
	Then $\Soc(A)=0$ because $\Sym_3$ has a trivial center.
\end{exa}

Naturally, one can quotient out skew braces by ideals to produce new skew braces.
Using the map $\lambda$ from Lemma~\ref{lambda} one shows that $I$ is a normal
subgroup of the additive group of $A$ (see Lemma~\ref{lem:ideals}) and that
for every $a \in A$ we have $a \circ I = a + I$. Then it follows that $A/I$ is a 
skew brace.

\begin{exa}
    Let $A=B_{8,5}$. This is the only classical brace of size eight with 
    additive group is isomorphic to $C_8$ and
    multiplicative group isomorphic to $C_4\times C_2$. It has four ideals which are
    isomorphic to $0$,  $B_{2,1}$, $B_{4,1}$ and $B_{8,5}$. The quotients of $A$ are then
    isomorphic to $B_{8,5}$, $B_{4,2}$, $B_{2,1}$ and $0$.
\end{exa}


A \emph{left ideal} $I$ of $A$ is a subgroup $I$ of the additive group
of $A$ such that $\lambda_a(I)\subseteq I$ for all $a\in A$.

\begin{exa}
	Let $A=B_{6,2}$, the only classical brace of size six
	with additive and multiplicative
	groups isomorphic to $C_6$. It has four left ideals which 
	are isomorphic to $0$, $B_{2,1}$, $B_{3,1}$ and $B_{6,2}$
	(in fact, all of them are two sided ideals in $A$).
\end{exa}

\begin{exa}
	Let $A=B_{6,1}$, the only non-trivial classical brace of size six
	with additive group isomorphic to $C_6$ and multiplicative
	groups isomorphic to $S_3$. It has one left ideal of size two,
	which is not a two-sided ideal.
\end{exa}

%

Let $I$ and $J$ be ideals
of a skew brace $A$. Then $I\cap J$ is an ideal of $A$. 
The sum 
$I+J$ of $I$ and $J$ is defined as the
additive subgroup of $A$ 
generated by all the 
elements of the form
$u+v$, $u\in I$ and $v\in J$. 

\begin{lem}
Let $A$ be a skew brace and let
$I$ and $J$ be ideals of $A$. Then $I+J$ is an ideal of $A$.
\end{lem}

\begin{proof}
    Let $a\in A$, $u\in I$ and $v\in J$. Then $\lambda_a(u+v)\in I+J$ and
    hence it follows that $\lambda_a(I+J)\subseteq I+J$. Moreover, 
    \[
        (u+v)*a=(u\circ\lambda^{-1}_u(v))*a
        =u*(\lambda^{-1}_u(v)*a)+\lambda^{-1}_u(v)*a+u*a\in I+J.
    \]
    This formula implies that  
    \[
        a\circ (u+v)\circ a'=a+\lambda_a((u+v)+(u+v)*a')-a\in I+J.
    \]
    Thus it follows that $a\circ (I+J)\circ a'\subseteq I+J$.
    
    Finally $I+J$ is a normal subgroup of $(A,+)$ since
    \[
        a+\left(\sum_{k} u_k+v_k\right)-a=\sum_k ((a+u_k-a)+(a+v_k-a))\in I+J
    \]
    whenever $u_k\in I$ and $v_k\in J$ for all $k$. 
\end{proof}

For $a,b\in A$ we write
$a*b=\lambda_a(b)-b$. For subsets $X$ and $Y$ of 
$A$ we write $X*Y$ to denote the subgroup
of $(A,+)$ generated by $\{x*y:x\in X,y\in Y\}$. 

\begin{exa}
Let $A=B_{8,18}$. It has three ideals which are isomorphic to $0$, $B_{4,3}$ or
$A$. Let $I$ be the ideal isomorphic to $B_{4,3}$. Since $A$ has no ideals of
size two, the subset $A*I$ of size two cannot be an ideal of $A$.
\end{exa}

\section{Series of ideals}
\label{series}

Following Rump \cite{MR2278047}, one defines 
the \emph{left series} of a skew brace $A$
recursively by $A^1=A$ and $A^{n+1}=A*A^n$ for $n\geq1$. Each $A^n$ 
is a left ideal of $A$. The following example
shows that in general $A^n$ is not a normal subgroup
of the additive group of $A$:

\begin{exa}
Let $A=S_{36,191}$. The left series of $A$ is $A^1=A$, $A^2\simeq S_{18,22}$
and $A^3\simeq B_{3,1}$. Then the additive group $A^3$ is not normal 
in the additive group of $A$. Indeed, the additive group of $A$
contains no normal subgroup of order three. 
\end{exa}

Similarly the \emph{right series} of $A$ is defined by $A^{(1)}=A$ and
$A^{(n+1)}=A^{(n)}*A$ for $n\geq1$. Each $A^{(n)}$ is an ideal of $A$.  A skew
brace $A$ is said to be \emph{left nilpotent} (resp. \emph{right nilpotent}) if
$A^n=0$ (resp. $A^{(n)}=0$) for some $n\in\N$.  See~\cite{BCJO2}
or~\cite{MR2278047} for examples.

\begin{exa}
    Let $A=B_{16,73}$. Up to isomorphism, the ideals of $A$ 
    are
    \[
    0,B_{2,1},B_{4,1},B_{4,2},B_{4,3},B_{8,10},B_{8,13},B_{8,19},B_{16,73}.
    \]
    Let $I$ be the ideal isomorphic to $B_{8,10}$. Then
    $I*I$ is a subset of size two which is not an ideal of $A$.
\end{exa}

%
%
%
%
%
%
%
%
%

\subsection*{Simple skew braces}

Recall that a skew brace $A$ is said to be \emph{simple} if its only ideals are $\{0\}$
and $A$. Simple skew braces are intensively studied, in particular simple
classical braces~\cite{MR3763276,BCJO2}. 

\begin{exa}
	Let $A=S_{12,22}$.  Then $(A,+)\simeq\Alt_4$ and $(A,\circ)\simeq
	C_3\rtimes C_4$. The skew brace $A$ is a simple skew brace and
	$A=A^n=A^{(n)}$ for all $n\in\N$. 
\end{exa}

\begin{exa}
	Let $A=S_{24,50}$. Then $(A,+) \simeq \SL_2(3)$ and $(A,\circ) \simeq
	C_3\rtimes C_8$. Furthermore $A=A^n=A^{(n)}$ for all $n\in\N$. This skew
	brace is not simple since for example $\Soc(A)\simeq B_{2,1}$.
\end{exa}

Let us count how many simple classical braces appear in our database.
It is known that classical braces of prime-power size are not simple.
Computer calculations show the following results:

\begin{pro}
	Let $A$ be a simple brace of order $n$, where $ 1 \le n \le 127$ and $n \ne 96$. 
	Then $A$ is isomorphic to $B_{24,94}$ or $B_{72,475}$.
\end{pro}

For skew braces we can prove the following proposition:

\begin{pro}
	Let $A$ be a simple skew brace of order $n$, where $ 1 \le n \le 63$ and $n \not \in \{32, 48, 54\}$. 
	Then $A$ is isomorphic to $S_{12,22}$, $S_{12,23}$, $S_{24,853}\cong B_{24,94}$ or to
	one of the skew braces $S_{60,k}$, where $145 \le k \le 152$, which are the
	only skew braces with additive group isomorphic to $\Alt_5$.
\end{pro}

\begin{question}
	\label{question:simple2sided}
    Are there simple two-sided skew braces of nilpotent type?
\end{question}

\section{Prime ideals and prime skew braces}
\label{prime}

At the conference ``Groups, rings and the Yang--Baxter equation'', Spa, 2017,
Louis Rowen suggested that it could be interesting to study prime ideals of skew
braces. 

\begin{defn}
	A skew brace $A$ is said to be \emph{prime} if for all
	non-zero ideals $I$ and $J$ one has $I*J\ne0$.   
\end{defn}

Simple non-trivial skew braces are prime.  The converse does not hold:

\begin{exa}
	The skew brace $A=S_{24,708}$ is not simple and it is prime. The additive
	group of $A$ is not nilpotent since it is isomorphic to $\Sym_4$.
\end{exa}

We found several examples of non-simple prime skew braces; in all cases the
additive group is not nilpotent.  Therefore it seems natural to ask the
following questions:

\begin{question}
	Let $A$ be a finite prime skew brace of nilpotent type. Is 
	$A$ simple?
\end{question}

\begin{question}
	Let $A$ be a finite classical prime brace. Is $A$ simple?
\end{question}

%

\begin{question}
	\label{question:prime2sided}
Are there prime two-sided skew braces of nilpotent type?
\end{question}


\begin{defn}
A skew brace $A$ is said to be \emph{semiprime} if for each non-zero ideal $I$
of $A$ one has $I*I\ne 0$. 
\end{defn}

Of course, prime skew braces are semiprime. The converse does not hold:

\begin{example}
	Let $A=S_{12,22}$. Since $A$ is a simple skew brace, $A$ is prime. The
	direct product $A\times A$ is semiprime and not prime.
\end{example}

\begin{defn}
We say that an ideal $I$ of a skew
brace $A$ is \emph{prime} (resp. semiprime) if $A/I$ is a prime (resp.
semiprime) skew brace. 

\end{defn}

In non-commutative ring theory there is a strong connection between prime
ideals and the Baer radical of  the ring. Recall that the Baer radical of a
ring $R$ (also called the prime radical) equals the intersection of all prime
ideals in $R$. 
Solvable and Baer radicals were also considered for non-associative algebras, loop algebras and semigroups by Amitsur in~\cite{MR0050563,MR0059256,MR0059257}. 
Below, we generalize some classical results which hold for rings to
skew braces.  
Our definitions are similar to
those of ring theory but not identical.

Let $A$ be a skew brace and $a\in A$.  By $\langle a\rangle$ we will denote the
smallest ideal of $A$ which contains $a$ (i.e., the ideal generated
by $a$ in $A$).

\begin{defn}
	Let $A$ be a skew brace. We say that  $ a_{1}, a_{2}, a_{3}, \ldots \in A$ is
	an \emph{$n$-sequence} if $a_{i+1}\in \langle a_{i}\rangle*\langle
	a_{i}\rangle$ for $i\geq1$.
\end{defn} 

\begin{defn}
A skew brace $A$ is said to be \emph{Baer radical} if for each $a\in A$, every
$n$-sequence starting with $a$ reaches zero at some point. An ideal $I$ of $A$
is said to be Baer radical if every $n$-sequence in $A$ starting with an element in $I$ reaches zero.
\end{defn}

\begin{lem}
	\label{lem:trick}
	Let $A$ be a skew brace and $J$ be an ideal in $A$. Let
	$a,b\in A$ such that $a-b\in J$ and $c\in \langle a\rangle$. Then there exists
	$c'\in \langle b\rangle$ such that $c-c'\in J$.
\end{lem}

\begin{proof}
	It follows from using the canonical map $A\to A/J$.	
\end{proof}

\begin{lem}
	\label{lem:another_trick}
	Let $A$ be a skew brace and let $I$ and $J$ be ideals of $A$.  Let
	$a_1,a_2,\dots$ be an $n$-sequence such that $a_k\in I+J$ for all $k$. Then
	there exist an $n$-sequence $i_1,i_2,\dots$ in $I$ and $j_1,j_2\dots\in J$
	such that $a_k=i_k+j_k$ for all $k$.
\end{lem}

\begin{proof}
	We proceed by induction on the length $l$ of the $n$-sequence
	$a_1,a_2,\dots a_l$. The case $l=1$ is trivial, so let us assume that the
	result holds for some $l\geq1$. Since $a_{l+1}\in \langle a_l\rangle
	*\langle a_l\rangle$, there exist $c_i,d_i\in\langle a_l\rangle$ such that
	$a_{l+1}=\sum c_i*d_i$. By applying Lemma~\ref{lem:trick} with $a=a_l$,
	$b=i_l$ and $c=c_i$ or $c=d_i$, there exist $c_i'\in \langle i_l\rangle$
	and $d_i'\in\langle i_l\rangle$ such that $c_i-c_i'\in J$ and $d_i-d_i'\in
	J$. Let $i_{l+1}=\sum c_i'*d_i'\in \langle i_l\rangle*\langle
	i_l\rangle\subseteq I$. Then $\pi(a_{l+1}-i_{l+1})=0$, where $\pi\colon A\to
	A/J$ is the canonical map.  This implies that $a_{l+1}-i_{l+1}\in J$ and
	the lemma follows.
\end{proof}

 \begin{lem}
	 \label{7} 
	 Let $A$ be a  skew brace.  The sum of any number of Baer radical ideals in
	 $A$ is a Baer radical ideal in $A$. 
 \end{lem}

\begin{proof}
 Let $I$ and $J$ be two Baer radical ideals in $A$. Since every ideal is a normal
 subgroup of the additive group of $A$, $I+J=\{i+j: i\in I, j\in J\}$. Consider
 an $n$-sequence $a_{1}, a_{2}, \ldots $ starting with an element $a_{1}=i+j$
 where $i\in I, j\in J$. 
 By Lemma~\ref{lem:another_trick}, $a_{m}\in J$ for some $m$.
 Now, since $J$ is  Baer radical, every $n$-sequence starting with $a_{m}$ will
 reach zero, therefore the $n$-sequence  $a_{m}, a_{m+1}, a_{m+2}, \ldots $
 will reach zero, as required.  Similarly, the sum of any number of Baer radical
 ideals is an ideal (as any element in this sum belongs to a sum of a finite
 number of these ideals).  
 \end{proof}

Lemma~\ref{7} implies that the sum of all Baer radical ideals in $A$ is the
largest Baer radical ideal in $A$. Thus a skew brace $A$ contains largest Baer
radical ideal.  This justifies the following definition:

\begin{defn}
	Let $A$ be a skew brace. The \emph{Baer radical} $B(A)$ of $A$ is the
	largest Baer radical ideal of $A$. 
\end{defn}

\begin{lem}\label{8} 
	Let $A$ be a skew brace and let $I$ be a  Baer radical ideal in $A$. 
	If $I$ and $A/I$ are Baer radical skew braces, then $A$ is a Baer radical
	skew brace.
\end{lem}

\begin{proof}  
	Let $a_{1}, a_{2}, \ldots $ be an $n$-sequence in $A$.  Because $A/I$ is
	Baer radical we get that $a_{m}\in I$ for some $m$.   Now, since $I$ is a
	Baer radical ideal, every $n$-sequence starting with $a_{m}$ will reach
	zero. Therefore the $n$-sequence  $a_{m}, a_{m+1}, a_{m+2}, \ldots $ will
	reach zero.
\end{proof}

\begin{lem}\label{9}
 Let $A$ be a skew brace, and $J$ be an ideal in $A$, and let $I$ be an ideal
 in the skew brace $A/J$.  Then $\bar {I}=\{a\in A: a+J\in I\}$ is an ideal in
 $A$.
\end{lem}

\begin{proof} 
 Note that $a\in\bar I$ if and only if $a+J\in I$. Let $a, b\in \bar
 {I}$. Since $a+J\in I$ and $b+J\in I$, $a+b+J=(a+J)+(b+J)\in I$. Hence $a+b\in \bar
 {I}$. Similarly, if $a\in \bar I$ and $c\in A$, then $a+J\in I$. Therefore
 $\lambda _{c}(a)+J=\lambda _{c+J}(a+J)\in I $ and hence $\lambda _{c}(a)\in \bar
 {I}$. Observe also that $c'\circ a\circ c+J=  (c'+J)\circ (a+J)\circ (c+J)\in
 I$, therefore $c'\circ a\circ c\in \bar{I}$.
\end{proof}

\begin{thm} 
\label{thm:B(A/B(A))=0}
Let $A$ be a skew brace. Then $B(A/B(A))=0$. 
\end{thm}

\begin{proof}
 By Lemma \ref{9}, 
 if the Baer radical $B(A/B(A))$ of $A/B(A)$ is nonzero, then 
 \[
	 I=\{a\in A: a+B(A)\in B(A/B(A))\}
 \]
 is an ideal of $A$. Notice that $B(A)\subseteq I$ and $I/B(A)=B(A/B(A))\neq 0$.
 Since $I/B(A)$ and $B(A)$ are Baer radical, 
 by Lemma \ref{8} one obtains that $I$ is a Baer radical ideal in $A$ and thus 
 $I\subseteq B(A)$. Hence $I/B(A)=0$, a contradiction. 
\end{proof}

\begin{lem}\label{12}
 If $J\subseteq I$ are ideals in a skew brace $A$, then $I/J$ is an ideal in
 $A/J$.
\end{lem}

\begin{proof}
Let $a+J,b+J\in I/J$ and $c\in A$. Then $a\in I$ and $b\in I$. Moreover, since $J$
is an ideal, $(a+J)+(b+J)=a+b+J\in I/J$, $\lambda _{c+J}(a+J)=\lambda
_{c}(a)+J\in I/J$ and  $(c+J)'\circ (a+J)\circ (c+J)=  c'\circ a\circ c+J\in
I/J$.
\end{proof}

\begin{pro}
 Let $A$ be a skew brace and $I, J$ be ideals in $A$, then $(I+J)/J$ is an
 ideal in $A/J$.
\end{pro}

\begin{proof}
 Since $I+J$ is an ideal, the claim follows from Lemma \ref{12}.
\end{proof}

\begin{lem}\label{1112}
	Let $A$ be a skew brace such that $B(A)\neq 0$. Then there is a non-zero
	ideal $I\subseteq B(A)$ in $A$ such that $I*I=0$.
\end{lem}

\begin{proof} 
	Let $a\in B(A)$. We construct an $n$-sequence of elements of $A$ starting
	with $a$.  Suppose that we defined elements $a_{1}, a_{2}, \ldots , a_{i}$
	of our sequence and they are all non-zero.  If $\langle
	a_{i}\rangle*\langle a_{i}\rangle$ is nonzero, we can add a non-zero
	element $a_{i+1}$ to this $n$-sequence.  Since $a\in B(A)$, every
	$n$-sequence starting with $a$ will reach zero.  Therefore there exists $j$
	such that $a_{j}\neq 0$ and $\langle a_{j}\rangle *\langle a_{j}\rangle=0$.
	Now take $I=\langle a_{j}\rangle$. Since $I\ne0$ and $I\subseteq B(A)$, the
	lemma is proved.
\end{proof}

\begin{thm}
	\label{thm:B(A)=0<=>semiprime}
Let $A$ be a skew brace. Then $A$ is semiprime  if and only if the Baer
radical of $A$ is zero. 
\end{thm}

\begin{proof}
   If $B(A)\ne0$, then the claim follows from Lemma~\ref{1112}.  Conversely,
   assume that $B(A)=0$ and that $A$ is not semiprime. Then there is a non-zero
   ideal $I$ such that $I*I=0$. Since every $n$-sequence starting with elements
   from $I$ reaches zero at the second place, it follows that $0\ne I\subseteq
   B(A)=0$. Since this is a contradiction, $A$ is semiprime. 
\end{proof}


\begin{thm} 
	\label{thm:Baer}
	Let $A$ be a skew brace. Then $B(A)$ equals the intersection of all prime
	ideals of $A$.
\end{thm}

\begin{proof} 
 Let $I$ be the intersection of all prime ideals in $A$. Then $I$ is an ideal
 of $A$. To prove that $I\subseteq B(A)$ we need to show that every
 $n$-sequence starting with any element of $I$ reaches zero. Let $a_{1}\in I$
 and $a_{1}, a_{2}, \ldots $ be an $n$-sequence.  Suppose on the contrary that
 this $n$-sequence contains only non-zero elements.  Let $J$ be a maximal ideal
 which does not contain any element from this $n$-sequence (it may be the zero
 ideal) and let $\pi\colon A\to A/J$ be the canonical map. 
 Note that every ideal in $A/J$ is of the form $\pi(L)$ for
 some ideal $L$ in A.  We claim that $J$ is a prime ideal. Indeed, if $P$ and
 $Q$ are ideals of $A$ properly containing $J$, the maximality of $J$ implies
 that there are $n,m\in\N$ such that $a_n\in P$ and $a_m\in Q$. Hence there
 exists $N\geq \max\{n,m\}$ such that $a_N\in P\cap Q$.  Since $0\ne a_{N+1}\in
 P*Q$ and $a_{N+1}\not\in J$, the non-zero ideals $\pi(P)$ and $\pi(Q)$ are
 such that $\pi(P)*\pi(Q)\ne0$. Therefore $J$ is prime and hence  $I\subseteq
 J$, a contradiction.

It remains to show that the Baer radical of $A$ is contained in every prime
ideal in $A$. Suppose on the contrary, let $P$ be a prime ideal in $A$ such
that $P$ does not contain $B(A)$. Then the factor brace $A/P$ has an element
$a+P\neq 0+P$ such that $a\in B(A)$.  We construct an $n$-sequence of elements
of $A$ starting with element $a\in B(A)\setminus P$.  Suppose that we defined
elements $a_{1}, a_{2}, \ldots , a_{i}\notin P$ of our sequence and they are
all non-zero.  Observe that if $\langle a_{i}\rangle*\langle a_{i}\rangle$ is
not a subset of $P$, then we can add a non-zero element $a_{i+1}\notin P$ to
this $n$-sequence.  Since $a\in B(A)$ then every $n$-sequence starting with $a$
will reach zero, therefore every $n$-sequence starting with $a$ will reach an
element in $P$. Therefore, there is $j$ in our $n$-sequence such that
$a_{j}\notin P$ and $\langle a_{j}\rangle*\langle a_{j}\rangle\subseteq P$.
Note that since $a_{1}\in B(A)$ then $a_{2}, a_{3}, \ldots , a_{j}\in B(A)$.
By Lemma \ref{12}, $L=\langle a_{j}+P\rangle/P$ is an ideal in $A/P$.  Note that
$L*L=0$ hence $A/P$ is not a prime skew brace, a contradiction since by
assumption $P$ is a prime ideal in $A$.
\end{proof}

\begin{cor}\label{dobra} 
	Let $A$ be a skew brace.  Then $A$ is Baer radical if and only if
	$A$ has no prime ideals except $A$.  In other words every $n$-sequence in a
	skew brace $A$ has zero element if and only if $A$ has no prime ideals
	except $A$.  
\end{cor}

\begin{proof}
	It follows from Theorem~\ref{thm:Baer}
\end{proof}

\begin{thm}
\label{thm:subdirect}
 Every semiprime skew brace embeds as a skew brace in a direct product of prime skew braces.
\end{thm}

\begin{proof}
  Let $A$ be a skew brace and $\{P_{i}:i\in T\}$ be the set of its prime ideals.
  Then $B(A)=\bigcap _{i\in T}P_{i}$. Consider the
  skew brace $Q$ which is  direct product of skew braces $A/P_{i}$, for $i\in T$.
  Consider the map $f:A\rightarrow Q$ where  $f(a)=\{a+P_{i}\}_{i\in T}$, then
  this is a homomorphism of skew braces. Observe, that the kernel of this map $f$
  equals the set of these elements which are in all prime ideals of $A$, hence
  it equals $B(A)$. It follows that the kernel of $f$ is zero. 
\end{proof}

\begin{defn}
 Let $A$ be a skew brace and let $I$ be an ideal of $A$. We say that an ideal $I$ in
 $A$ is a left (resp. right) nilpotent ideal if $I$ is a left nilpotent (resp.
 right nilpotent) skew brace.
 \end{defn}

 \begin{defn}
	 Let $A$ be a skew brace. The \emph{Wedderburn radical} $W(A)$ of $A$ is
	 defined as the sum of all the ideals of $A$ that are left nilpotent and
	 right nilpotent.
 \end{defn}

\begin{lem}
\label{16}
Let $A$ be a two-sided brace. Then $W(A)\subseteq B(A)$.
\end{lem}

\begin{proof} 
If $I$ is a nilpotent ideal, then $I\subseteq B(A)$ since every $n$-sequence
reaches zero. The result now follows from Lemma \ref{7}. It also follows from noncommutative ring theory since every two-sided brace is a ring.
\end{proof}

\begin{thm}
	\label{thm:B(A)=0<=>W(A)=0}
 Let $A$ be a two-sided brace. Then $B(A)=0$ if and only if $W(A)=0$.
\end{thm}

\begin{proof}
 By Lemma \ref{16}, if $B(A)=0$ then $W(A)=0$. Suppose that $B(A)\neq 0$ then
 $W(A)\neq 0$ by Lemma \ref{1112}. It also follows from noncommutative ring theory since every two-sided brace is a ring.
\end{proof}

\begin{cor}
 A two-sided brace with a non-zero Baer radical is not prime and not semiprime. In
 particular, a finite skew brace which is either left nilpotent or right
 nilpotent is not prime and not semiprime.
\end{cor}

\begin{proof}  
	The first assertion follows from Lemma~\ref{1112}.
	The second
	assertion follows from the fact that if a two-sided brace $A$ is either left or
	right nilpotent, $A\subseteq B(A)$. 
\end{proof}

It is known that in finite rings the Wedderburn and the Baer radical are equal.
This does not happen for infinite rings. Therefore, for infinite skew braces,
the Baer and the Wedderburn radical are not in general equal. This follows from
the following lemma:

\begin{lem}
	\label{15}
	Let $(A, +,\cdot)$ be a Jacobson radical ring and $(A, +, \circ )$ be the
	associated two-sided brace (this means that $a\circ b=a+b+a\cdot b$ for all
	$a,b\in A$).  Then the Baer radical of the brace $(A, +, \circ )$ equals
	the Baer radical of the ring $(A, +, \cdot)$. 
\end{lem}

\begin{proof}
	It follows from the fact that the intersection of prime ideals in any ring
	equals the Baer radical of this ring. So the Baer radical of the Jacobson
	radical ring $(A, +,\cdot)$ will be equal to the intersection of all its
	prime ideals.  This equals the intersection of all prime ideals in the
	corresponding brace $(A, +, \circ)$. By Theorem~\ref{thm:Baer}, this is
	equal to the Baer radical of the brace $(A, +, \circ )$. To finish the
	argument, one uses a result of Ced\'o, Jespers and
	Okni\'nski~\cite[Proposition 1]{MR3177933} stating that every ideal  $I$ in
	a two-sided brace $A$ comes from the associated Jacobson radical ring and
	gives a two-sided factor brace $A/I$.  
\end{proof}

The Baer and the Wedderburn radical might be different even in the case of
finite skew braces: 

\begin{exa}
		Let $A=S_{6,2}$ be the unique non-trivial skew brace with additive and
		multiplicative group isomorphic to $\Sym_3$. Then $W(A)\simeq B_{3,1}$
		and $B(A)=A$. 
\end{exa}

\section{Solvable ideals}
\label{solvable}

Motivated by the theory of groups, Bachiller, Ced\'o, Jespers and Okni\'nski
introduced solvable braces~\cite{BCJO2}. The definition not only works in the
case of classical braces. For a skew brace $A$ we define $A_{1}=A$ and
inductively $A_{i+1}=A_{i}*A_{i}$ for $i\geq1$. Recall that $A$ is said to be
\emph{solvable} if $A_{n}=0$ for some $n$.  By induction one proves that
$A_{i+1}\subseteq A_{i}$ for all $i$.  An ideal $I$ in a skew brace $A$ is
\emph{solvable}, if $I$ is a solvable skew brace.  Clearly every finite solvable two-sided 
brace is Baer radical as every $n$-sequence will reach zero.  

\begin{lem}
	Let $A$ be a skew brace. For each $j$, $A_{j+1}$ is an ideal of $A_{j}$. In
	particular, each $A_j$ is a sub skew brace of $A$.
\end{lem}

\begin{proof}
	It follows since $A_{j+1}=A_j*A_j=(A_j)^{(2)}$ for all $j$. 
\end{proof}

\begin{example}
	Let $A=B_{48,396}$. Then 
	$A_1=A$, $A_2\simeq B_{24,58}$, $A_3\simeq B_{6,1}$, $A_4\simeq B_{3,1}$ and $A_5=0$.
	The sub skew brace $A_3$ is not an ideal.
\end{example}

The aim of this section is to show that for finite skew braces the Baer radical
equals the largest solvable ideal.  For that purpose, we need some preliminary
results. Some of these results were proved by Ced\'o, Jespers and Okni\'nski
in~\cite{BCJO2} for classical braces. 

\begin{lem}
	\label{umsolvable} 
	A sum of a finite number of solvable
	ideals in a skew brace is solvable.
\end{lem}

\begin{proof} 
	Let $I$ and $J$ be solvable ideals in $A$, 
	$T=I+J$, $T_{1}=T$ and $T_{n+1}=T_{n}*T_{n}$ for $n\geq1$. Similarly, let $I_{1}=I$ 
	and $I_{n+1}=I_{n}*I_{n}$ for $n\geq1$.  Notice that $I_{m}=0$ and $J_{m'}=0$  for some $m,
	m'$ since $I$ and $J$ are solvable. It can be proved by induction that for every
	$i$, $T_{i}\subseteq J_{i}+I$ (by showing that $T_{i}/I\subseteq (J_{i}+I)/I$ in the
	skew brace $A/I$). It follows that $T_{m'}\subseteq I$, and therefore
	$T_{m+m'}=0$. Therefore a sum of two solvable ideals is solvable. By using
	induction on the number of ideals we can show that sum of any finite number
	of solvable ideals is solvable.
\end{proof}

\begin{lem}
	\label{factor} 
	Let $A$ be a skew brace and let $I$ be a solvable ideal in $A$. If $A/I$ is
	a solvable skew brace, then $A$ is a solvable skew brace.
\end{lem}

\begin{proof} 
	Denote inductively  $T=A/I$, $T_{1}=T, T_{n+1}=T_{n}*T_{n}$, $I_{1}=I,
	I_{n+1}=I_{n}*I_{n}$.  Notice that $I_{m}=0$ and $T_{m'}=0$  for some $m,
	m'$ since $I$ and $T$ are solvable. 
	It can be proved by induction that, for every $i$,
	$A_{i}+I =T_{i}$ in $A/I$. Therefore $A_{m'}\subseteq I$. Consequently
	$A_{m+m'}\subseteq I_{m}=0$.
\end{proof}

\begin{lem}
	\label{weddern}
 Let $A$ be a  finite skew brace. Then the Wedderburn radical of $A$ is a
 solvable ideal in $A$. 
\end{lem}

\begin{proof} 
	Let $I$ be a left and right nilpotent ideal of $A$, it can be shown by
 induction that $I_{n}\subseteq I^{n}$ and $I_{n}\subseteq I^{(n)}$, therefore
 $I$ is solvable. Our result now follows from  Lemma \ref{umsolvable}.
\end{proof}

Now we are ready to prove the main result of the section:

\begin{thm}
\label{main}
 Let $A$ be a finite two-sided brace. Then $A$ is Baer radical if and only if $A$ is
 solvable.
\end{thm}

\begin{proof} 
	Clearly every finite solvable two-sided brace is a Baer radical
	skew  brace (by induction on cardinality and Lemmas~\ref{factor} and~\ref{8} with $I=A*A$).
	Suppose now that $A$ is a Baer radical skew
	brace, so $A\subseteq B(A)$. We will prove that $A$ is solvable by
	induction on the number of elements in $A$. If $A$ has only one element
	then $A$ is a trivial brace and the result holds. Suppose the result holds
	for all skew braces of cardinality smaller than $i$, and suppose that $A$
	has cardinality $i+1$. Since $B(A)\neq 0$ we get $W(A)\neq 0$ (by
	Theorem~\ref{thm:B(A)=0<=>W(A)=0}).  By Lemma \ref{weddern}, $W(A)$ is a solvable ideal in
	$A$. Since $A$ is Baer radical it follows that $A/W(A)$ is  Baer radical.
	By the inductive assumption  $A/W(A)$ is solvable. By Lemma \ref{factor}
	applied for $I=W(A)$ we get that $A$ is solvable. 
\end{proof} 

\begin{cor}
	Let $A$ be a finite two-sided brace and $I$ be an ideal in $A$. Then $I$ is solvable if
	and only if $I$ is Baer radical.  In particular, the Baer radical of $A$ equals
	the largest solvable ideal.  Moreover, if $A$ has a nonzero solvable ideal then
	$A$ has a non-zero ideal $I$ such that $I*I=0$.
\end{cor}

\begin{proof}
	Every Baer radical ideal in $A$ is solvable, hence $B(A)$ is solvable. On the other
 hand every solvable ideal is Baer radical by Theorem \ref{main}. The rest
 follows from Lemma \ref{1112}. 
\end{proof}

\begin{cor}
	There exists an infinite brace such that its Baer radical is not solvable.
\end{cor}

\begin{proof}
	It follows from the fact that there are Baer radical rings which are not
	nilpotent.
\end{proof}

\begin{rem}
Question~\ref{question:prime2sided} asks for two-sided braces that are prime.  Note
that there exists an infinite prime and not simple Jacobson radical ring.  Thus,
by Lemma~\ref{15}, there exists a prime and not simple infinite two-sided
brace.
\end{rem}

The results in this section allow us to answer the following question: Is it
true that a product of any number of non-zero ideals in $A$ (in any order) is
nonzero? 

\begin{lem}
\label{intersection}
Let $A$ be a prime skew brace and let $I$ and $J$ be non-zero ideals in $A$. Then
 $I\cap J$ is a non-zero ideal in $A$. Moreover, the intersection of any finite
 number of non-zero ideals in $A$ is a non-zero ideal in $A$.
\end{lem}

\begin{proof}
	The intersection of any two ideals is an ideal. Notice that $I* J\subseteq
	I\cap J$, therefore $I\cap J\neq 0$.  The last assertion can be proved by
	induction on the number of ideals. 
\end{proof}

\begin{lem}
\label{PPP}
 Let $A$ be a semiprime skew brace with no non-zero solvable ideals and $I$ be a non-zero ideal in $A$. Then the 
 product of any number of copies of $I$, multiplied in any order, is non-zero.
 Moreover, any product of copies of $I$ contains some $I_{n}$. 
\end{lem}

\begin{proof}
 We use an induction on $i$, the number of copies of $I$ used in our product. If
 $i=1$ then our product equals $I=I_{1}\neq 0$. Suppose now that any product of
 any number of at most $i$ copies of $I$ contains $I_{n}$ for some $n$. Let $P$
 be a product of $i+1$ copies of $I$, for some $i>0$. Then $P=P_{1}*P_{2}$
 where $P_{1}$ and $P_{2}$ are products of at most $i$ copies of $I$. By the
 inductive assumption, $I_{n}*I_{n}=I_{n+1}\subseteq P_{1}*P_{2}=P$ for some
 $n$. 
Notice that $I_{n}\neq 0$ for every $n$. Indeed if $I_{n}=0$ then  $I$ is
solvable. \end{proof}

\begin{thm}
 If $A$ is a prime skew brace with no non-zero solvable ideals, then a product of any number of non-zero ideals,
 multiplied in any order, in $A$ is non-zero.
\end{thm}

\begin{proof} 
	Denote our product of ideals as $P$.  Let $I_{1}, \ldots , I_{m}$ be ideals
	used in the product $P$. By Lemma \ref{intersection}, $T=\bigcap
	_{k=1}^{m}I_{k}$ is a nonzero ideal in $A$. Let $Q$ be a product of copies
	of ideal $T$ obtained by exchanging any ideal among $I_{1}, \ldots , I_{m}$
	appearing in the product $P$ by ideal $T$. Clearly $Q\subseteq P$.  Note
	that $A$ is semiprime since it is prime. By Lemma~\ref{PPP}, $Q\neq
	0$.
\end{proof}

\section*{Acknowledgements}

We thank Emiliano Acri, Ferran Ced\'o, Rofoldo
Cossalter, Louis Rowen and Ivan Shestakov for comments and corrections.

\section*{Funding}

\textcolor{red}{The first-named author is partially supported by CCP CoDiMa (EP/M022641/1)
and the OpenDreamKit Horizon 2020 European Research Infrastructures project (\#676541).
The second-named author is supprted by the 
European Research Council Advanced grant 320974. 
The third-named author is supported by PICT-201-0147, MATH-AmSud 17MATH-01 and
European Research Council 
Advanced grant 320974.}

\bibliographystyle{abbrv}
\bibliography{refs}

\end{document}